\documentclass[10pt]{article}
\usepackage[usenames]{color}
\usepackage{amssymb}
\usepackage{multicol}
\usepackage{graphicx}
\usepackage{amsthm}
\usepackage{verbatim}
\usepackage{float}
\usepackage{forest,adjustbox}
\usepackage{color}
\usepackage{mathrsfs}
\usepackage{amsmath}
\usetikzlibrary{arrows,decorations.pathmorphing,backgrounds,positioning,fit,petri,calc}
\usepackage{amsmath, epsfig}

\theoremstyle{plain}
\newtheorem{thm}{Theorem}
\newtheorem{lemma}[thm]{Lemma}
\newtheorem{cor}[thm]{Corollary}

\newtheorem{conj}[thm]{Conjecture}
\newtheorem*{Directed Path Partition Conjecture (DPPC)}{Directed Path Partition Conjecture (DPPC)}

\newtheorem*{ques*}{Question}

\title{Extended Path Partition Conjecture for Semicomplete and Acyclic Compositions}
\author{
Jiangdong Ai\thanks{Department of Computer Science. Royal Holloway University of London.  {\tt Jiangdong.Ai.2018@live.rhul.ac.uk}.} \and Stefanie Gerke\thanks{Department of Mathematics. Royal Holloway University of London.  {\tt stefanie.gerke@rhul.ac.uk}.} \and Gregory Gutin \thanks{Department of Computer Science. Royal Holloway University of London. {\tt g.gutin@rhul.ac.uk}.} \and Yacong Zhou\thanks{Department of Computer Science. Royal Holloway University of London. {\tt Yacong.Zhou.2021@live.rhul.ac.uk}.} }
\begin{document}
 \maketitle

 \begin{abstract}
Let $D$ be a digraph and let $\lambda(D)$ denote the number of vertices in a longest path of $D$. For a pair of vertex-disjoint induced subdigraphs $A$ and $B$ of $D$, we say that $(A,B)$ is a partition of $D$ if $V(A)\cup V(B)=V(D).$
The Path Partition Conjecture (PPC) states that for every digraph, $D$, and every integer $q$ with $1\leq q\leq\lambda(D)-1$, there exists a partition $(A,B)$ of $D$ such that $\lambda(A)\leq q$ and $\lambda(B)\leq\lambda(D)-q.$ Let $T$ be a digraph with vertex set $\{u_1,\dots, u_t\}$ and for every $i\in [t]$, let $H_i$ be a digraph with vertex set $\{u_{i,j_i}\colon\,
 j_i\in [n_i]\}$.
The {\em composition} $Q=T[H_1,\dots , H_t]$  of $T$  and $H_1,\ldots, H_t$ is a digraph with vertex set $\{u_{i,j_i}\colon\,  i\in [t],  j_i\in [n_i]\}$ and arc set
$$A(Q)=\cup^t_{i=1}A(H_i)\cup \{u_{i,j_i}u_{p,q_p}\colon\, u_iu_p\in A(T), j_i\in [n_i], q_p\in [n_p]\}.$$ We say that $Q$ is acyclic {(semicomplete, respectively)} if $T$ is acyclic {(semicomplete, respectively)}. In this paper, we introduce a conjecture stronger than PPC using a property first studied by Bang-Jensen, Nielsen and Yeo (2006)
and show that the stronger conjecture holds for wide families of acyclic and semicomplete compositions.
 \end{abstract}

\section{Introduction}\label{sec:intro}
Let $G$ be a directed or undirected graph. The {\em order} of $G$ is the number of vertices in $G.$ Let $\lambda(G)$ be the maximum order of a path in $G.$ (In this paper all paths and cycles in digraphs are directed.)
A {\em partition} of $G$ is a pair $(A,B)$ of vertex-disjoint induced subgraphs of $G$ such that $V(A)\cup V(B)=V(G).$  {Terminology and notation for digraphs not introduced in this section follows those in \cite{BG,DiClasses}. }

\vspace{1mm}

The {\em Path Partition Conjecture} (PPC) can be stated as follows.

\begin{conj}\label{PPC}[Path Partition Conjecture (PPC)]
Let $G$ be a directed or undirected graph. For every natural number $q$ with $0<q<\lambda(G)$ there is a partition $(A,B)$ of $G$ such that $\lambda(A)\leq q$ and $\lambda(B)\leq \lambda(G)-q.$
\end{conj}

Note that the PPC trivially holds for every directed and undirected graph $G$ with a Hamilton path $P$ as  one can choose any induced subgraph on $q$ vertices  for $A$.  In general, the PPC remains far from proven.
Frick \cite{FM} mentioned that the PPC for undirected graphs was discussed by Lov\'{a}sz and Mih\'{o}k in $1981$ in Szeged. Several results supporting the conjecture for undirected graphs have been proved in the literature, see e.g. \cite{BI,BI2,BU,H,HW,V}. In 1982, Laborde et al. \cite{LC} extended Conjecture \ref{PPC} to digraphs. (Note that the PPC for undirected graphs is equivalent to the PPC for symmetric digraphs; a digraph $D$ is {\em symmetric} if it can be obtained from an undirected graph $G$ by replacing every edge $uv$ by the pair $uv,vu$ of arcs.) There is a number of papers in which the PPC is proven for special classes of digraphs, see e.g. \cite{AG,BMA,F,LC,WW}.

In this paper we consider a stronger notion which was introduced in \cite{BMA} by Bang-Jensen, Nielsen and Yeo which we call the BNY property. To state the property we need the following notation.
A {\em $k$-path subdigraph} of a digraph $D$ is a collection of $k$ vertex-disjoint paths in $D$. The maximum order of a $k$-path subdigraph $D$ will be denoted by $\lambda_k(D)$; clearly $\lambda(D)=\lambda_1(D).$

\vspace{2mm}

\noindent{\bf Property BNY.} {\em A digraph $D$ satisfies {\em BNY} if for every integer
$q$ with $1\le q\le \lambda(D)-1$ there is a partition $(A,B)$ of $D$ such that $\lambda(A)\leq q$ and   for all $k\in \{1,2,\dots , |V(B)|\}$, $\lambda_k(B)\leq \lambda_k(D)-q.$}

\vspace{2mm}

We conjecture that one can strengthen the PPC in the following way.

\begin{conj}\label{BYNC}
Property BNY holds for every digraph.
\end{conj}

 Bang-Jensen, Nielsen and Yeo \cite{BMA} proved that Conjecture~\ref{BYNC} holds for several generalizations of tournaments, see also Corollary~\ref{corBNY} in Section~\ref{S4}. In this paper we show that Conjecture~\ref{BYNC} holds for some compositions of digraphs. Let $T$ be a digraph with vertex set $\{u_1,\dots, u_t\}$ and  for every $i\in [t]$,  let $H_i$ be a digraph with vertex set $\{u_{i,j_i}\colon\,
 j_i\in [n_i]\}$.
The \emph{composition} of $T$ and $H_1,\dots, H_t$
is the digraph  $Q=T[H_1,\dots , H_t]$ with vertex set $\{u_{i,j_i}\colon\,  i\in [t],  j_i\in [n_i]\}$ and arc set
$$A(Q)=\cup^t_{i=1}A(H_i)\cup \{u_{i,j_i}u_{p,q_p}\colon\, u_iu_p\in A(T), j_i\in [n_i], q_p\in [n_p]\}.$$
See Figure \ref{fig:Q} for an example of a composition, $Q=T_3[H_1, H_2, H_3]$ where $V(H_1)=\{a,d\}$, $V(H_2)=\{b\}$, $V(H_3)=\{c\}$ and $T_3$ is a 3-cycle.

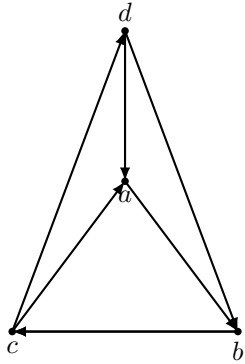
\begin{figure}
\begin{center}
\begin{tikzpicture}[>=stealth,thick]
\fill (0,1) circle (0.05);
\fill (0,3) circle (0.05);
\fill (1.5,-1) circle (0.05);
\fill (-1.5,-1) circle (0.05);
\node [below] at (0,1) {$a$};
\node [below] at (1.5,-1) {$b$};
\node [below] at (-1.5,-1) {$c$};
\node [above] at (0,3) {$d$};
\draw [-latex] (0,1)--(1.5,-1);
\draw [-latex] (1.5,-1)--(-1.5,-1);
\draw [-latex] (-1.5,-1)--(0,1);
\draw [-latex] (0,3)--(1.5,-1);
\draw [-latex] (1.5,-1)--(-1.5,-1);
\draw [-latex] (-1.5,-1)--(0,3);
\draw [-latex] (0,3)--(0,1);

\end{tikzpicture}
\end{center}\caption{$Q=T_3[H_1, H_2, H_3]$}\label{fig:Q}
\end{figure}

A composition $Q$ is {\em acyclic} if $T$ is acyclic, that is, $T$ contains no directed cycles.
Acyclic digraphs form a well-studied family of digraphs of great interest in graph theory, algorithms and applications (e.g. see Chapter 3 of \cite{DiClasses}).
 The main result of \cite{AG} states that if
 $Q=T[H_1,\dots , H_t]$ is an acyclic composition and the PPC holds for every digraph $H_i$, $i \in [t]$, then the  PPC holds for $Q$. We show that analogous result for the BNY property, see Section~\ref{S3}.

\begin{thm}\label{T5}
Let $Q=T[H_1,\dots,H_t]$ be an acyclic composition.
If  for every $i\in [t]$,  $H_i$ satisfies property BNY, then $Q$ satisfies property BNY.
\end{thm}

Note that a single vertex satisfies the property BNY and it follows from Theorem~\ref{T5} that every acyclic digraph satisfies property BNY.

We also consider semicomplete compositions. A digraph $D$ is {\em semicomplete} if for every pair $x,y$ of distinct vertices of $D$, there is at least one arc between $x$ and $y.$ In particular, a tournament is a semicomplete digraph, where there is exactly one arc between $x,y$ for every pair $x,y$ of distinct vertices. A composition $Q=T[H_1,\ldots,H_t]$ is called \emph{semicomplete} if $T$ is a semicomplete digraph.  We show (a strengthening of) the following theorem in
Section \ref{S4}.

\begin{thm}\label{BNYsemicomplete}
Let $Q=T[H_1,\dots,H_t]$ be a semicomplete composition.
If  for every $i\in [t]$,  $H_i$ satisfies property BNY, then $Q$ satisfies property BNY.
\end{thm}

\section{Acyclic Compositions}\label{S3}

 We restate Theorem \ref{T5} for the convenience of the reader. Note that the main idea of the proof of Theorem \ref{T5} is similar to that of the main acyclic composition result of \cite{AG}.

\medskip

\noindent{\bf Theorem \ref{T5}}
{\em Let $Q=T[H_1,\dots,H_t]$ be an acyclic composition.
If  for every $i\in [t]$,  $H_i$ satisfies property BNY, then $Q$ satisfies property BNY.
}

\begin{proof}

For $i=1,\ldots,t$, define $p_i^{in}$ as the maximum order of a path in $Q-H_i$ whose terminal vertex has an arc into $H_i$, and $p_i^{end}$ as the maximum order of a path in $Q$ such that the terminal vertex of this path belongs to $H_i$. Note that $T$ is acyclic and thus no path  can enter and leave $H_i$ more than once and hence $p_i^{end}=p_i^{in}+\lambda(H_i)$ and in particular $p_i^{end}>p_i^{in}$.

We partition $Q$ as follows
\begin{enumerate}
		\item If $p_i^{end}\leq q$, then the entire graph $H_i$ belongs to $A$.
		\item If $p_i^{in}\geq q$, then the entire graph $H_i$  belongs to  $B$.
		\item If $p_i^{in}<q<p_i^{end}$, we let $q_i=q-p^{in}_i$ and  partition $H_i$ into $(A_i,B_i)$ such that for all $k=1,\ldots, V(B_i)$, we have $\lambda(A_i)\leq q_i$ and $\lambda_k(B_i)\leq \lambda_k(H_i)-q_i$. The digraph  $A_i$ belongs to $A$ and  the digraph  $B_i$ belongs to $B$.
	\end{enumerate}

We first show that $\lambda(A)\leq q$. Consider a longest path $P$  in $A$ and let $x\in V(H_i)$ be the last vertex of this path. If $p_i^{end} \leq q$ then by definition of $p_i^{end}$ and the fact that $A$ is a subgraph of $Q$, the longest path ending in $x$ has order at most $q$ and thus $P$ has order at most $q$.  If $p_i^{end} >q$ then $p_i^{in}<q$ and the longest path in the intersection of $A$ and $H_i$ has order at most $q-p_i^{in}$. As $T$ is acyclic $P$ can enter $H_i$ at most once and thus $P$ has order at most $p_i^{in}+q-p_i^{in}=q$.

Now we show that for $k=1,\ldots, |V(B)|$,   $\lambda_k(B)\leq \lambda_k(Q)-q$.
Let $W_k=P_1\cup P_2\cdots\cup P_k$ be a $k$-path subdigraph with maximum order in $B$, and let $x_1$, $x_2$, $\dots$, $x_k$ be the first vertices of $P_1$, $P_2$, $\dots$, $P_k$, respectively. Without loss of generality, we assume that $x_1$ is in the graph $H_l$ that has the minimum $p^{in}$-value among all the graphs with starting vertices in $W_k$. Suppose that $P'$ is a  path of order $p_l^{in}$ in $Q-H_l$, such that the terminal vertex  of $P'$ has an arc into $H_l$. We will show that $P'$ does not intersect $W_k$. In fact we claim something  stronger namely that  for every path  from $x\in V(H_i)$ to $y\in V(H_j)$, either $i=j$ or $p_i^{end} \leq  p_j^{in}$. Since there is a path from every vertex in $P'$ to $x_1$ and $P'$ avoids $H_l$, it follows that the $p^{in}$-values of all their host graphs $H_i$ must be smaller than $p_l^{in}$.

To prove our claim assume that $i\not=j$ and consider a longest path $P^*$ of length $p_i^{end}$ with a terminal vertex in $H_i$. This path $P^*$ cannot use a vertex in $H_j$ as by the definition of compositions and the existence of a path from $x\in V(H_i)$ to $y\in V(H_j)$ there is a path from every vertex in $H_i$ to every vertex in $H_j$ and thus (a part of) $P^*$ together with a path from $H_i$ to $H_j$ would form a closed walk and therefore would induce a cycle in $T$. Hence we can extend $P^*$ to a path to $H_j$ and thus $p_i^{end}\leq p_j^{in}$.

 Now, if  $p_l^{in} \geq q$ then $|V(P')|=p_l^{in}\geq q$. We can obtain a new $k$-path subdigraph $W_k'\subseteq Q$ from $W_k$ by appending $P_1$ to $P'$. Observe that $\lambda_k(Q)\geq |W_k'|= |W_k|+p_l^{in}\geq \lambda_k(B)+q$.
	If $p_l^{in} < q$  then $p_l^{end} > q$. Assume that $W_k$ intersects $B_l$ in $b$ paths. By construction, we know $\lambda_b(B_l)\leq \lambda_b(H_l)-q_l$. We substitute the $b$ paths in $B_l$ by a maximum $b$-path subdigraph in $H_l$ and adding $P'$ to one of the paths. This gives
	\begin{align*}
	\lambda_k(Q)&\geq |W_k|+(\lambda_b(H_l)-\lambda_b(B_l))+|V(P)| \\
	& \geq \lambda_k(B)+q_l+p_l^{in}\\
	& =   \lambda_k(B)+q,
	\end{align*}
	which completes the proof.
\end{proof}

\section{Semicomplete Compositions}\label{S4}

Recall that composition $Q=T[H_1,\dots , H_t]$ is semicomplete if $T$ is a semicomplete digraph, that is, for  every pair $x,y$ of distinct vertices in $T$, there is at least one arc between $x$ and $y$.   In addition, if $H_1,\dots,H_t$ are digraphs without arcs, we call $Q$ an {\em extended semicomplete digraph} and denote it by $Q=T[E_{n_1},\dots, E_{n_t}]$ where $n_i=|V(H_i)|$ for $1\leq i\leq t.$

Bang-Jensen et al. \cite{BMA} proved the following somewhat surprising result.
\begin{lemma}\label{l1}\cite{BMA}
Let $Q'=T[E_{n_1},\dots, E_{n_t}]$ be an extended semicomplete digraph and let $v_{i,k}$ denote the maximum number of vertices of $E_{n_i}$ that can be covered by a $k$-path subdigraph in $Q'$. Then for each $i\in [t]$ every maximum $k$-path subdigraph in $Q'$ covers exactly $v_{i,k}$ vertices of $E_{n_i}.$
\end{lemma}

Note that in the previous lemma $v_{i,k}$ is defined as the maximum number of vertices in $E_{n_j}$ in  \emph{any} $k$-path subdigraph  not only the maximal ones. In particular this implies that every semicomplete graph has a  Hamilton path, which is a much easier statement of course.
With Lemma~\ref{l1} we can prove the following theorem which is a generalisation of a result in \cite{BMA} and Sun \cite{Y}. Our proof is similar to that in  \cite{Y}.

\begin{thm}\label{T4}
Let $Q=T[H_1,\dots,H_t]$ be a semicomplete composition and let $l_{i,k}$ denote the maximum number of vertices of $H_i$ that can be covered by a $k$-path subdigraph in $Q$. Then for each $i\in [t]$ every maximum $k$-path subdigraph in $Q$ covers exactly $l_{i,k}$ vertices of $H_i$.
\end{thm}
\begin{proof}
Let $Q'=T[E_{n_1},\dots, E_{n_t}]$ be the extended semicomplete digraph obtained by removing all arcs in $H_i$ for $1\leq i\leq t$ and define $v_{i,k}$ as in Lemma \ref{l1}.
According to Lemma \ref{l1}, every maximum $k$-path subdigraph $L$ in $Q'$ contains exactly $v_{i,k}$ vertices from $E_{n_i}.$ For each $i=1,\dots,t$ consider a (maximum) $v_{i,k}$-path subdigraph $L_i.$ Now  for each $i\in [t]$, substitute the  $v_{i,k}$ vertices of $E_{n_i}$ in $L$ by a path from $L_i$ to obtain a $k$-path subdigraph in $Q$ containing $|L_i|$ vertices in $H_i$. Note that this operation is possible as each vertex in $H_i$ has the same in- and out-neighbours outside $H_i.$
 Thus $l_{i,k}\geq |L_i|.$

We will now show that $l_{i,k}\leq |L_i|$ and thus $l_{i,k}=|L_i|.$ Assume for a contradiction that there exists a $k$-path subdigraph $P_k$ in $Q$ that covers in, say $H_1$, more than $|L_1|$ vertices. Note that $P_k$ can enter or visit $H_1$ at most $v_{1,k}$ times because $P_k$ translates to a $k$-path subdigraph in $Q'$ (for example by replacing each path in $H_1$ by the first vertex of the corresponding path in $E_{n_1}$.) Thus the graph induced by $P_k$ in $H_1$ is a $b$-paths subdigraph of $H_1$ with $b\leq v_{1,k}$ and thus this graph has no more than $|L_1|$ vertices.

As $|L_i|=l_{i,k}$ no $k$-path subdigraph in $Q$ can have more than $|L_1|+\ldots+ |L_t|$ vertices and as we have constructed a (maximum) path with $|L_1|+\ldots+ |L_t|$ vertices, it follows that in fact every maximum  $k$-path subdigraph covers exactly $l_{i,k}$ vertices in $H_i$.
\end{proof}

Now we can prove the following theorem, which  is a strengthening of Theorem~\ref{BNYsemicomplete}.

\begin{thm}\label{T6}
Let $Q=T[H_1,\dots,H_t]$ be a semicomplete composition. Let $l_i$ denote the maximum number of vertices of $H_i$ that can be covered by a longest path in $Q$, and let $l_m=\max \{l_i\mid i\in [t]\}.$
If $H_m$ satisfies BNY, then $Q$ satisfies BNY as well.
\end{thm}

\begin{proof}
Define $l_{i,k}$ as in Theorem \ref{T4} as the maximum number of vertices of $H_i$ that can be covered by a $k$-path subdigraph in $Q$. Note that $l_i=l_{i,1}$ and since every $k$-paths subigraph of $H_i$ is also a $k$-path subdigraph in $Q$ we have

\begin{equation}  \lambda_k(H_i)\leq l_{i,k}. \label{e1}\end{equation}
Moreover by Theorem \ref{T4}, we have
\begin{equation}
         \lambda_k(Q)= l_{1,k}+l_{2,k}+\dots+l_{t,k}. \label{q1}
\end{equation}

In particular,  $\lambda(Q)= l_1+\ldots+ l_t$.
We distinguish between 3 cases depending on $q$ and $l_i,\ldots, l_t$. The first case is not strictly necessary but it shows the approach without any technical difficulties. The third case  is the most interesting and there we need that  $H_m$ satisfies BNY.

\noindent{\bf Case 1:}  $q=l_{i_1}+ \ldots + l_{i_s}$ for some indices $i_1,\ldots,i_s \in \{1,\ldots,t\}$.
Then $A$ consists of the graph induced by the vertex set of  $H_{i_1}, H_{i_2},\ldots, H_{i_s}$ and $B$ of the graph induced by the remaining vertices.  Note that any path in $A$ is also a path in $Q$ and thus at most $l_{i_j}$ vertices of $H_{i_j}$ can belong to any path in $A$ and particularly to a longest path. Thus
\[
\lambda(A)\leq l_{i_1}+\dots+l_{i_s} = q \]
and similarly
\[
\lambda_k(B) \leq \sum_{j \not\in \{i_1,\ldots, i_s\}} l_{j,k}    \stackrel{(\ref{q1})}{=}\lambda_k(Q)-(l_{i_1,k}+\dots+l_{i_s,k}).
\]
Note that $l_{i,k+1} \geq l_{i,k}$ as one can obtain a $(k+1)$-path subgraph from a $k$-path subdigraph by deleting an edge (or if the $k$-path subdigraph consists of isolated vertices then adding another vertex.). In particular
\begin{equation} \label{lik} l_{i,k} \geq l_i \end{equation} and thus
\[    \lambda_k(B)     \leq\lambda_k(Q)-(l_{i_1}+\dots+l_{i_s}) \leq \lambda_k(Q)-q. \]

\noindent{\bf Case 2:}  $l_{i_1}+ \ldots + l_{i_s}< q < l_{i_1}+ \ldots+ l_{i_s}+ l_h$ for some indices $i_1,\ldots,i_s,h \in \{1,\ldots,t \}$ and $l_h=|H_h|$.
In this case $A$ consists of the graph induced by $H_{i_1},\ldots H_{i_s}$ and (any) $q'= q-l_{i_1} - \ldots -l_{i_s}$ vertices of $H_h$. The graph $B$ will be induced by the remaining vertices and in particular will have $l_{h}-q'$ vertices of $H_h$. Again, any path in $A$ is also a path in $Q$ and thus at most $l_{i_j}$ vertices of $H_{i_j}$ can belong to any path in $A$ and particularly to a longest path. Thus
\[
\lambda(A) \leq q'+ l_{i_1} +\dots+l_{i_s}  =q.
\]

For $B$ we note that
$|H_h| \geq l_{h,k} \geq l_h = |H_h|$  and thus
\begin{eqnarray}
\lambda_k(B)&\leq& l_{h,k} - q' +  \sum_{j \not\in \{i_1,\ldots, i_s, h\}} l_{j,k}\nonumber \\
            &\stackrel{(\ref{q1})}{=}& l_{h,k} - q'+ \lambda_k(Q)-(l_{i_1,k}+\dots+l_{i_s,k}+l_{h,k}) \nonumber \\
            &\leq& \lambda_k(Q)-q'-(l_{i_1,k}+\dots+l_{i_s,k}) \nonumber \\
             &\stackrel{(\ref{lik})}{\leq}& \lambda_k(Q)-q'-(l_{i_1}+\dots+l_{i_s}) \nonumber \\
           &\leq& \lambda_k(Q)-q.\nonumber
\end{eqnarray}

\noindent{\bf Case 3:} $l_{i_1}+ \ldots + l_{i_s} < q <  l_{i_1}+ \ldots + l_{i_s}+ l_m$ and for all $j\in \{i_1,\ldots,i_s, m\}$ we have $l_j<|H_j|$.

We first prove that the subgraph of the semicomplete graph $T$ that is induced by the vertices corresponding to $H_{i_1},\ldots, H_{i_s}, H_m$ is acyclic. So assume this graph has a cycle including the vertex corresponding  to $H_{j_1}\ldots H_{j_p}$.  For each $r\in \{1,\ldots,p\}$, any path visits $H_{j_r}$ at most $l_{j_r}<|H_{j_r}|$ times in $Q$ and therefore can be extended along the cycle avoiding the vertices on this path. In particular a longest path in $Q$ can be extended which means it was not the longest after all.

Note that since the subgraph of the semicomplete graph $T$ that is induced by the vertices corresponding to $H_{i_1},\ldots, H_{i_s}, H_m$ is acyclic, any  (longest) path in this graph will visits $H_m$ (and all other $H_{i_j}$) only once.

Now let $q' := q- l_{i_1}- \ldots -l_{i_s}$.
 Let $(A_m,B_m)$ be a partition of $H_m$ such that $\lambda(A_m) \leq q' $ and for all $k'\in \{1,\ldots, |V(B_m)|\}$, we have $\lambda_{k'}(B_m) \leq \lambda_{k'}(H_m) - q'$.
 Let $A$ be the graph induced by the vertices of $H_{i_1},\ldots, H_{i_s}, A_m$ and let $B$ be the graph induced by the remaining vertices. Because the underlying subgraph of $T$ induced by $A$ is acyclic as we discussed above, we have
\begin{eqnarray}
\lambda(A)&\leq & \lambda(H_{i_1})+ \ldots + \lambda(H_{i_s})+\lambda(A_m) \nonumber\\
          &\stackrel{(\ref{e1})}{\leq}& l_{i_1}+\ldots+ l_{i_s}+ q'  \nonumber\\
         &=& q.\nonumber
\end{eqnarray}

 Let $Q'=T[E_{n_1},\dots, E_{n_t}]$ be the extended semicomplete digraph obtained by removing all arcs in $H_i$ for $1\leq i\leq t$. Fix $k\in \{1,\ldots, |V(B)|\}$. Define $v_{i,k}$ as in Lemma \ref{l1} as the maximum number of vertices of $E_{n_i}$ that can be covered by a $k$-path subdigraph in Q'. Note that no $k$-path subdigraph in $Q$ can visit $H_m$ more than $v_{m,k}$ times, since each $k$-path subdigraph in $Q$ visiting $H_m$ more than $v_{m,k}$ times can be translated to a $k$-path subdigraph in $Q'$ containing more than $v_{m,k}$ vertices in $E_{n_m}$. Let $u_k=\min\{ v_{m,k},|V(B_m|)\}$.
If $u_k\geq 1$ then
\begin{eqnarray} \label{uk}
\lambda_{u_k}(B_m) &\leq& \lambda_{u_k}(H_m) -q'\leq  l_{m,k}-q'.
\end{eqnarray}
If $u_k=0$ (that is, if $\lambda(H_m)\leq q'$ and we can choose $A_m=H_m$ and $B_m$ to be empty) we set $\lambda_{u_k}= l_{m,k}-q'$.

Any (maximum) $k$-path subdigraph in $B$ is also a $k$-path subdigraph in $Q$,
thus any $k$-path subdigraph in $B$ can visit $B_m$ at most $u_k$ times. Therefore
\begin{eqnarray*}
 \lambda_k(B)  &\leq& \lambda_{u_k}(B_m)+ \sum_{j \not\in \{i_1,\ldots, i_s, m\}} l_{j,k} \\
             &\stackrel{(\ref{uk})}{\leq}& l_{m,k} - q' +  \lambda_k(Q) -  \sum_{j \in \{i_1,\ldots, i_s, m\}} l_{j,k} \\
           &\stackrel{(\ref{lik})}{\leq} & \lambda_k(Q) - q'-  \sum_{j \in \{i_1,\ldots, i_s\}} l_{j} \\
              &\leq& \lambda_k(Q)-q.
\end{eqnarray*}
\end{proof}

A digraph is {\em strong} if there is a directed path between any two vertices.
A digraph $D$ is {\em quasi-transitive}, if the existence of the arcs $xy$ and $yz$ implies an arc between $x$ and $z$.
A digraph is {\em locally in-semicomplete}, if the in-neighbours of every vertex induce a semicomplete digraph. Important results on the above three classes are overviewed in chapters of \cite{DiClasses}.

Strong semicomplete compositions generalize strong quasi-transitive digraphs \cite{BH}.
However, semicomplete compositions form a more complicated digraph class than quasi-transitive digraphs. For example, it is well known that the Hamilton Cycle  problem (HCP) is polynomial-time solvable for quasi-transitive digraphs \cite{GGG}. However, HCP is  $\mathcal{NP}$-complete for semicomplete compositions.  Indeed, it is well known that the Hamilton Path  problem is $\mathcal{NP}$-complete for arbitrary digraphs.
Let $D$ be a digraph. Form a new digraph $D'$ from $D$ by adding a new vertex $u$ with all arcs from $u$ to $D$ and a new vertex $v$ with all arcs from $D$ to $v$ and an arc $vu.$ Observe that $D'$ has a Hamilton cycle if and only if $D$ has a Hamilton path.
Thus, HCP is $\mathcal{NP}$-complete even for semicomplete compositions in which $T$ is a 2-cycle. We refer the readers to \cite{Y} for the structural properties of semicomplete compositions and some results on connectivity, paths, cycles, strong spanning subdigraphs and acyclic spanning subgraphs.

By Theorems \ref{T5} and \ref{T6} and  results in \cite{BMA} on BNY, we have the following:

\begin{cor} \label{corBNY}
Let $Q=T[H_1,\dots,H_t]$ be an acyclic or semicomplete composition such that each $H_i$ is a quasi-transitive digraph, an extended semicomplete digraph, a locally in-semicomplete digraph, or an acyclic digraph.
Then $Q$ satisfies BNY.
\end{cor}

\end{document}